\documentclass[12pt]{article}

\usepackage{amsfonts,epsfig,amsmath,amsthm,amssymb,graphics,verbatim,overpic}
\usepackage[letterpaper,margin=0.85in]{geometry}

\newtheorem{thm}{Theorem}[section]

\newtheorem{prop}[thm]{Proposition}
\newtheorem{lem}[thm]{Lemma}
\newtheorem{ex}[thm]{Example}

\newtheorem{defn}[thm]{Definition}

\def\R{\mathbb{R}}

\def\N{\mathbb{N}}
\def\E{\mathbb{E}}
\def\P{\mathbb{P}}
\def\I{\infty}

\newcommand{\be}{\begin{equation}}
\newcommand{\ee}{\end{equation}}
\newcommand{\bea}{\begin{eqnarray}}
\newcommand{\eea}{\end{eqnarray}}
\newcommand{\beann}{\begin{eqnarray*}}
\newcommand{\eeann}{\end{eqnarray*}}
\newcommand{\benn}{\begin{equation*}}
\newcommand{\eenn}{\end{equation*}}

\def\ra{\rightarrow}
\def\I{\infty}
\def\Id{{\textnormal{Id}}}

\newcommand{\cB}{{\mathcal B}}  
\newcommand{\cC}{{\mathcal C}}  
\newcommand{\cE}{{\mathcal E}}  
\newcommand{\cG}{{\mathcal G}}  
\newcommand{\cH}{{\mathcal H}}  
\newcommand{\cI}{{\mathcal I}}  
\newcommand{\cJ}{{\mathcal J}}  
\newcommand{\cK}{{\mathcal K}}  
\newcommand{\cM}{{\mathcal M}}  
\newcommand{\cO}{{\mathcal O}}  
\newcommand{\cT}{{\mathcal T}}  
\newcommand{\cV}{{\mathcal V}}  
\newcommand{\cX}{{\mathcal X}}  
\newcommand{\cY}{{\mathcal Y}}  
\newcommand{\cZ}{{\mathcal Z}}  

\def\txtd{{\textnormal{d}}}
\def\txte{{\textnormal{e}}}

\begin{document}

\author{Christian Kuehn\thanks{Technical University of Munich, Faculty of Mathematics,
Boltzmannstr.~3, 85748 Garching b.~M{\"u}nchen, Germany; e-mail: ckuehn@ma.tum.de}
\thanks{External Faculty, Complexity Science Hub Vienna, Josefst\"adterstrasse 39, A-1090 Vienna, Austria}}
 
\title{Multiscale Dynamics of an Adaptive Catalytic Network}

\maketitle

\begin{abstract}
We study the multiscale structure of the Jain-Krishna adaptive network model. This model describes the co-evolution of a set of continuous-time autocatalytic ordinary differential equations and its underlying discrete-time graph structure. The graph dynamics is governed by deletion of vertices with asymptotically weak concentrations of prevalence and then re-insertion of vertices with new random connections. In this work we prove several results about convergence of the continuous-time dynamics to equilibrium points. Furthermore, we motivate via formal asymptotic calculations several conjectures regarding the discrete-time graph updates. In summary, our results clearly show that there are several time scales in the problem depending upon system parameters, and that analysis can be carried out in certain singular limits. This shows that for the Jain-Krishna model, and potentially many other adaptive network models, a mixture of deterministic and/or stochastic multiscale methods is a good approach to work towards a rigorous mathematical analysis.      
\end{abstract}


\section{Introduction}

Catalytic reactions have been studied in a wide variety of 
contexts, classically in chemistry~\cite{LindstroemPettersson}, but more recently
across mathematics, physics and the life sciences. A particular area, where 
catalytic reactions have been employed is the origins of life problem, i.e., how
to form biology out of pre-biotic systems. One proposed mechanism is that various
molecules can catalyze each other to form a self-sustaining and self-organized 
reaction network, which is then able to form more complex structures. This paradigm
is linked to the notions of autocatalytic sets~\cite{Kauffman1,Kauffman,HordijkSteel}
as well as to hypercycles~\cite{Eigen,EigenSchuster}, which have both become 
standard ideas in complex systems~\cite{Gros}. 

A new twist has been added to this line of research by more recent models in
network science. In particular, autocatalytic reaction mechanisms have been paired
with the idea of adaptive (or co-evolutionary) networks~\cite{GrossSayama}, where
the dynamics \emph{of} the network is fully coupled to the dynamics \emph{on} 
the network. For example, adaptive networks have been
used as epidemic models~\cite{GrossDLimaBlasius,ShawSchwartz,HorstmeyerKuehnThurner}, for 
evolutionary games on networks \cite{PachecoTraulsenNowak,KuehnZschalerGross}, 
and for modelling self-organized criticality~\cite{BornholdtRohlf,KuehnNetworks}. 

In this work, we study an adaptive network model proposed by Jain and 
Krishna~\cite{JainKrishna,JainKrishna1,JainKrishna2,JainKrishna3}.
The Jain-Krishna model builds upon two well-established mathematical ideas: ordinary
differential equations (ODEs) for autocatalytic systems, and network structure 
formation models via random graph theory. Each vertex/node in the model is part of a graph/network 
on which ODE dynamics takes place. Each vertex has an associated concentration, which is
the variable for its ODE, which is coupled to the ODE of other vertices via the 
graph structure. The idea
is that once the concentrations in the ODEs have reached a suitable asymptotic limit 
set in forward time, we update the graph structure by deleting the vertex with the 
lowest asymptotic concentration and insert a new vertex with random connections to
the rest of the graph; see also Section~\ref{sec:JKmodel} for a precise definition of 
the model.

It seems natural to conjecture that \emph{multiscale analysis}
is going to play a key role for the Jain-Krishna model, and in fact almost all other
adaptive network models. For example, one may ask whether there is a natural time scale 
separation between the process \emph{of the network} relative to the dynamics \emph{on
the network}. In fact, there are even further time scales, such as the convergence
rate to invariant sets for the ODEs, or the time scale to reach a structurally different
graph via adaptation. If we can identify the different time scales in the
problem, then this would open up the entire methodology of multiple time scale 
dynamical systems methods such as geometric singular perturbation 
theory~\cite{Fenichel4,Jones,Tikhonov}; see also~\cite{KuehnBook} for a more detailed
background on multiple time scale methods. However, even if we have identified time
scales, the multiscale paradigm suggests to separate the scales first in suitable 
singular limits, i.e., fast dynamics for frozen slow variables or slow dynamics
for constrained/averaged fast variables. The problem only simplifies if these
singular limits are mathematically tractable. In this work, we make the following 
progress towards this multiscale program for the Jain-Krishna model:

\begin{itemize}
 \item[(R1)] We provide a proof that the autocatalytic ODEs have equilibrium points 
and solutions generically converge to an equilibrium exponentially fast. We also algebraically
characterize the sets of equilibrium points. 
 \item[(R2)] In the context of the ODE proof, we uncover a relation to projective spaces.
We also provide an intertwining lemma linking network topology to convergence 
structure.
 \item[(R3)] For the network adaptation rule, we formulate four conjectures\footnote{Although 
it may be evident to almost all readers, the author
would be very interested to see any proof of these conjectures, or counter-examples with 
improved conjectures/theorems in future works.} regarding the
relevant time scales based upon network size $d$ and the edge probability parameter $p$. 
In particular, we study the formation time of the first cycle as well as the formation time 
of a single autocatalytic set (ACS) encompassing the entire graph.
 \item[(R4)] Although we do not prove the four conjectures from (R3), we provide 
heuristic/formal asymptotic calculations to motivate the intrinsic scalings appearing in all 
four cases.   
\end{itemize}

In summary, we have advanced the understanding of the Jain-Krishna model as our results
clearly show a multiple time scale structure based upon the ODE convergence
time as well as the graph structure formation time for various ranges of $p$ and $d$. 
Furthermore, (R1)-(R2) show that the singular limit for a frozen graph structure is 
tractable. The formal calculations in (R3)-(R4) indicate that the graph adaptation process 
can be analyzed as well. However, to provide a full singular perturbation analysis in the
finite time-scale separation regime is beyond our approach here and remains a challenge for
future work. Numerical simulations strongly indicate~\cite{JainKrishna,JainKrishna1,JainKrishna2,
JainKrishna3} that it is possible to study the finite time-scale case. In fact, the numerical 
results show on a coarse-grained level typical fast-jump, slow-drift, relaxation-oscillation, 
or bursting structures commonly encountered for multiple time scale problems. These observations
were actually a main motivation to understand the apparent visual link to multiscale dynamics 
in more detail.\medskip  

The structure of the paper is as follows: In Section~\ref{sec:JKmodel} we define the Jain-Krishna
model; we also point out an ill-defined variant of it in Appendix~\ref{ap:bug}, which has been
used in the literature. In Section~\ref{sec:cts}
we study the continuous-time dynamics on the network given by a system of autocatalytic-reaction 
ODEs and prove results (R1)-(R2) regarding existence and stability of equilibrium points. In
Section~\ref{sec:dis} we consider the discrete-time dynamics of the network. We identify two
different regimes based upon giant-component existence or non-existence. Then we split these
two cases into (a) the first cycle formation problem, which is linked back to classical literature
on random graph dynamics, and (b) the formation problem of large autocatalytic sets. Based upon
these ideas we provide approximate and asymptotic calculations to motivate four conjectures on 
the discrete-time Jain-Krishna model graph update rule.\medskip

\textit{Acknowledgments: I would like to thank two anonymous referees, who helped to 
improve the presentation of the results. I also would like to thank the VolkswagenStiftung 
for support via a Lichtenberg Professorship. Furthermore, I acknowledge general interesting
discussions on adaptive networks with Leonhard Horstmeyer within a joint project
of the Austrian Science Foundation (FWF, Project No. P29252) as well as the support
of the Complexity Science Hub Vienna (CSH) by funding the workshop ``Adaptive 
Networks'' November 2-3 in 2017, which provided additional motivation for me to 
undertake the research presented here.}  

\section{The Jain-Krishna Model}
\label{sec:JKmodel}

The version of the Jain-Krishna model we follow can be found 
in~\cite{JainKrishna2}. Fix a dimension $d\in\N$, $d\geq 2$, representing 
the number of species in the reaction pool. Let $C=C[s]=(c_{ij}[s])=(c_{ij})$ 
for $s\in\N_0$ be a $d\times d$ matrix of interaction coefficients, where $s$ 
represents a discrete time step, which we shall, as indicated above, sometimes 
drop from the notation. We can view $C$ as the transpose of the adjacency matrix of 
a directed graph $\cG=\cG(\cV,\cE)$ with vertex set 
$\cV=\{1,2,\ldots,d\}$ and edge set $\cE=\{(i,j):c_{ij}=1\}$, i.e., 
if there is an edge from $j$ to $i$ then $c_{ij}=1$ and $c_{ij}=0$ 
otherwise\footnote{Classically, one would work with the transpose $C^\top$ as 
it is the classical adjacency matrix but to keep with the standard conventions for the 
Jain-Krishna model, we use $C$.}; self-loops are not allowed 
so that $c_{jj}=0$ for all $j\in\{1,2,\ldots,d\}$. Whenever we consider
a subgraph of $\cG'=\cG'(\cV',\cE')\subset \cG$, we not only require
$\cV'\subseteq \cV$ and $\cE'\subseteq\cE$ but consider induced subgraphs,
where start- and end-points of edges are contained in the subgraph. We 
denote the set of all possible allowed matrices $C$ as
\benn
\cM:=\left\{M=(m_{ij})\in\R^{d\times d}:m_{ij}\in \{0,1\},~
m_{jj}=0~\forall i,j\right\}.
\eenn
Let $t\in\R$ be a continuous time and consider the following ODEs 
for the values $x_j=x_j(t)$ at each vertex
\be
\label{eq:JK}
\frac{\txtd x_j}{\txtd t}=:
x_j'=(Cx)_j-x_j\sum_{k=1}^d (Cx)_k,\qquad j\in\{1,2,\ldots,d\},
\ee
with initial condition $x(0)$ and we set $x=(x_1,x_2,\ldots,x_d)^\top$; 
here $(\cdot)^\top$ denotes the transpose so we work with column vectors. 
We impose the two constraints
\be
\label{eq:JKc}
\sum_{k=1}^d x_j=1 \qquad \text{and}\qquad x_k\geq 0\quad \forall k,
\ee
which we refer to as mass and non-negativity conservation.
Let $\cX$ be the subset of $\R^d$ containing all $x\in\R^d$ for 
which~\eqref{eq:JKc} holds. We shall prove in 
Proposition~\ref{prop:invariant} that if the initial condition $x(0)$ 
satisfies~\eqref{eq:JKc}, then $x(t)$ satisfies~\eqref{eq:JKc} as well 
for any $t>0$. Note that~\eqref{eq:JK} can be written more compactly as
\be
\label{eq:JKvec}
x'=Cx-|Cx|_1x=:f(x),\qquad |x|_1:=\sum_{j=1}^d x_j,
\ee
which shall not lead to any confusion with the usual $1$-norm as 
all components of $x$ will always be non-negative. We shall prove
in Theorems~\ref{thm:convtosteady1} and \ref{thm:convpart2} that~\eqref{eq:JK}, 
for \emph{fixed} $C$, always converges for sufficiently generic initial 
data to an equilibrium point $x_*$ as $t\ra +\I$. 

As yet, the network $C$ is static. The model is 
turned into a fully adaptive (or co-evolutionary) network as follows. 
The first matrix $C[0]$ is sampled as an Erd\"os-Renyi-type
random graph, i.e., 
\be
\label{eq:sample}
\P(c_{ij}= 1)=p\qquad\text{and} \qquad \P(c_{ij}= 0)=1-p 
\ee
for $i\neq j$ and a fixed parameter $p\in(0,1)$. Assume that we converged 
to an equilibrium point 
$x_*$ for~\eqref{eq:JK} for $C=C[0]$. Then we define the following 
set of indices
\benn
\cJ_*:=\left\{j\in\{1,2,\ldots,d\}:x_j=\min_k x_k\right\},
\eenn 
which are the species with minimum prevalence. Next, pick some 
$j_*$ at random with equal probability from $\cJ_*$ and re-sample 
$c_{ij_*}$ and $c_{j_*i}$ for $i\neq j$ according to~\eqref{eq:sample}
keeping all the other entries fixed. This step corresponds to 
eliminating one of the species that has performed worst in the 
autocatalytic reaction process~\eqref{eq:JK}. This yields a new matrix 
$C=C[1]$. Now we can repeat the process, including a new equilibrium 
point $x_*=x_*[1]$. The adaptive network dynamical
system can formally be written as a mapping
\benn
\phi:\N_0\times \cX\times \cM\ra \cX\times \cM,\qquad 
\phi(s,x_*[0],C[0])=(x_*[s],C[s]). 
\eenn
It is already an interesting question under which assumptions/modifications
the process just described generates a random dynamical system in the
sense of~\cite{ArnoldSDE}. This question will be considered in future
work. Here we focus on the two basic components, the static ODE dynamics
and the graph structure. These components should be viewed as singular
limits of the fully adaptive network, which is usually simulated using
a finite-time scale for the ODE dynamics at each vertex.

\section{Results for the Continuous-Time Vertex Dynamics}
\label{sec:cts}

In this section, we study the ODEs~\eqref{eq:JK}-\eqref{eq:JKc}. Our main 
goal is to provide a detailed analysis\footnote{Good physical intuition and numerical
calculations exist already for the ODEs~\cite{JainKrishna,JainKrishna1,JainKrishna2,JainKrishna3}; 
however, it seems useful to have a detailed analysis of the mathematical structures
behind the ODE convergence.} of all possible cases, which can 
occur depending upon a fixed and given matrix $C$.

\subsection{Well-Posedness \& Consistency}

The next result shows that $\cX$ is positively invariant so that the two 
constraints are always satisfied.

\begin{prop}
\label{prop:invariant}
Suppose $x(0)\in\cX$, then $x(t)\in\cX$ for all $t>0$.
\end{prop}

\begin{proof}
Regarding mass conservation, we compute
\benn
\frac{\txtd }{\txtd t}\left(\sum_{j=1}^d x_j\right)=
\sum_{j=1}^d (Cx)_j-\sum_{j=1}^dx_j~\sum_{k=1}^d (Cx)_k=
\left(1-\sum_{j=1}^dx_j\right)\sum_{j=1}^d (Cx)_j
\eenn
and mass is indeed conserved if $\sum_{j=1}^dx_j(0)=1$. For 
non-negativity, we decompose $\partial\cX$ and define 
$\partial\cX_j:=\partial\cX\cap \{x\in\R^d:x_j=0\}$. Suppose $x\in 
\partial\cX_j$, then the $j$-th component of the 
ODEs~\eqref{eq:JK} is given by
\benn
x_j'=(Cx)_j\geq 0,
\eenn
where the last inequality follows as we start inside $\cX$ 
and since $C$ has non-negative entries. Since $j$ was arbitrary
and vector field either vanishes or points inside on 
$\partial\cX$, the non-negativity conservation follows.
\end{proof}

We remark that there are variants of the Jain-Krishna model in the 
literature, which are not well-posed as shown in Appendix~\ref{ap:bug}.
Hence, one has to be very careful, which variant of the model is 
discussed in various sources. 

\subsection{Existence of Equilibria} 
\label{ssec:steadystates}

Let $0:=(0,0,\ldots,0)^\top$ and observe that although $f(0)=0$, the origin 
is not an equilibrium point for~\eqref{eq:JK}-\eqref{eq:JKc} as $0\not\in \cX$. 
We briefly recall that for a general matrix $A=(a_{ij})\in\R^{m\times n}$ one 
calls $A$ positive and writes $A>0$ if $a_{ij}>0$ for all $i,j$. Similar 
component-wise definitions apply to non-negative matrices $A\geq 0$ and 
naturally specialize to the vector case $n=1$. The existence of a 
non-negative and non-zero equilibrium point $x_*$, such that 
\be
\label{eq:ssinX}
f(x_*)=0,\qquad \sum_{j=1}^d(x_*)_j=1, \qquad x_*\not\equiv 0, 
\qquad x_*\geq 0,
\ee
is already an interesting problem. Before analyzing it, recall that
a matrix $A$ is irreducible if and only if there exists a permutation
matrix $P$ such that $P^\top AP$ is block-upper triangular; if $A$ is 
not irreducible, it is called reducible. Furthermore,
viewing $A$ as the adjacency matrix of a directed graph $\cH$, then one
can easily prove that $A$ is irreducible if and only if $\cH$ is strongly
connected, i.e., if in $\cH$ there is a path from any vertex to any 
other vertex. The next two examples illustrate one of the major obstacles
to determine $x_*$.

\begin{ex}
\label{ex:base}
Let $d=3$ and consider the matrix $C$ with $c_{ij}=1$ for 
$(i,j)=(2,1),(1,2),(1,3)$ and $c_{ij}=0$ otherwise. Then one checks
that $x_*=(\frac12,\frac12,0)^\top$ is an equilibrium point. Note that $C$ 
is reducible but the subgraph formed by vertices $1$ and $2$ is irreducible.
\end{ex}

The last example suggests that we might want to use the Perron-Frobenius 
Theorem. Recall that there are several versions. The strong 
Perron-Frobenius states that for a matrix $A> 0$ there exists an 
eigenvalue $\lambda=\lambda(A)>0$ of algebraic multiplicity $1$ such that 
$\rho(A)=\lambda$, where $\rho(A)$ is the spectral radius of $A$. 
Furthermore, there exists a unique eigenvector $v$, called the Perron 
vector, such that 
\benn
Av=\lambda v,\qquad |v|_1=1,~v>0,
\eenn
and $v_1$ is the unique non-negative eigenvector up to positive 
multiples. Obviously, this version of the strong Perron-Frobenius Theorem 
is not applicable to the Jain-Krishna model as $C>0$ never holds. A 
more general version of the strong Perron-Frobenius Theorem states that 
the same conclusions apply if $A\geq 0$ and $A$ is irreducible.

\begin{ex}
\label{ex:base1}
Hence, applying the strong Perron-Frobenius Theorem in 
Example~\ref{ex:base} to the sub-graph/matrix formed by the first
two vertices, taking $((x_*)_1,(x_*)_2)=v_1^\top$ and augmenting the 
result by $(x_*)_3=0$ yields an equilibrium point. Yet, this naive
algorithm is flawed. Let $d=3$ and consider the matrix $C$ with 
$c_{ij}=1$ for $(i,j)=(2,1),(1,2),(3,1)$ and $c_{ij}=0$ 
otherwise so that we just reserved the edge between vertex $1$ 
and $3$ in comparison to Example~\ref{ex:base}. Then one checks
that $x_*=(\frac13,\frac13,\frac13)^\top$ is an equilibrium point and 
$x_*=(\frac12,\frac12,0)^\top$ is not an equilibrium point.
\end{ex}

Hence, applying the strong Perron-Frobenius Theorem naively to a 
reducible $C\geq 0$ or an irreducible sub-graph with appended zeros 
to obtain the existence of equilibrium points is not possible. The weak 
Perron-Frobenius Theorem states that if $A\geq 0$, then $\rho(A)$ is 
eigenvalue (not necessarily positive) and there exists an eigenvector 
$v\geq 0$, $v\not\equiv 0$.

\begin{thm}
\label{thm:exsteady}
Suppose $\cG$ has at least one edge, then there exists $x_*$
satisfying~\eqref{eq:ssinX}. 
\end{thm}

\begin{proof}
Applying the weak form of the Perron-Frobenius Theorem to $C$, we get
a non-zero and non-negative eigenvector $v$ with eigenvalue 
$\lambda=\rho(C)$. We define 
\benn
x_*=\frac{1}{|v|_1}v,
\eenn
so that $x_*\not\equiv 0$, $x_*\geq 0$, $|x_*|_1=1$. Then we compute
\benn
f(x_*)=Cx_*-x_*|Cx_*|_1=\frac{1}{|v|_1}Cv-\frac{1}{(|v|_1)^2}v|Cv|_1
=\frac{1}{|v|_1}\left(\lambda v-\lambda v \cdot 1\right)=0,
\eenn
which finishes the proof.
\end{proof}

Yet, the last result is not very constructive as we would like to know,
how many non-zero components equilibrium points really have. This information
is not provided by the weak form of Perron-Frobenius; 
cf.~Examples~\ref{ex:base}-\ref{ex:base1}. Furthermore, the dimension of
the set of equilibrium points is of interest as the next two examples show.

\begin{ex}
\label{ex:base2}
Let $d=4$ and consider the matrix $C$ with $c_{ij}=1$ for 
$(i,j)=(2,1)$, $(1,2)$, $(4,3)$, $(3,4)$ and $c_{ij}=0$ otherwise. Then 
$v_{1}=(\frac12,\frac12,0,0)^\top$ and $v_{2}=(0,0,\frac12,\frac12)^\top$ 
are Perron-Frobenius eigenvectors (PFEs) for the matrices with just $(i,j)=(2,1)$, 
$(1,2)$, respectively $(i,j)=(4,3)$, $(3,4)$ as non-zero indices splitting
the associated graph into its two $2$-cycles. One checks that 
$x_*=bv_{1}-(1-b)v_{2}$ is an equilibrium point for any $b\in[0,1]$.
\end{ex}

Therefore, we have shown that equilibrium points can be non-unique. Just using 
a general convex combination of Perron-Frobenius eigenvectors (PFEs) of irreducible 
subgraphs is also not enough.

\begin{ex}
\label{ex:ACSex1}
We use the same $C$ as in Example~\ref{ex:base2} with the added entry 
$c_{31}=1$. Then one checks that $x_*=bv_{1}-(1-b)v_{2}$ is an equilibrium point
only for $b=0$.
\end{ex}

\begin{defn}
An autocatalytic set is an (induced) subgraph $\cG'(\cV',\cE')\subset\cG$ 
such that for each $i\in \cV'$ there exists $j\neq i$ such that 
$(i,j)\in \cE'$.
\end{defn}

For completeness we record the next simple and well-known 
result~\cite{JainKrishna1} together with a short proof.

\begin{lem}
Let $C$ be the matrix associated to a directed graph $\cG$.
\begin{itemize}
\item[(L1)] We have the following implications
\be
\label{eq:relations}
\textnormal{$\cG$ is a cycle}\quad \Rightarrow \quad
\textnormal{$\cG$ is irreducible}\quad \Rightarrow \quad
\textnormal{$\cG$ is an ACS}. 
\ee
The converse implications are false. 
\item[(L2)] If $\cG$ has no cycle, then $\lambda(C)=0$. If $\cG$ has a
cycle, then $\lambda(C)\geq 1$.    
\item[(L3)] An ACS must contain a cycle. Furthermore, 
suppose $\lambda(C)\geq 1$ and let $v$ be the eigenvector associated 
to $\lambda(C)$. Then the subgraph with vertices $\{j:v_j>0\}$ is an ACS. 
\end{itemize}
\end{lem}

\begin{proof}
The implications~\eqref{eq:relations} in (L1) easily follow from 
the definitions, e.g., if $\cG$ is a cycle, then there exists a path 
from every vertex to every other by going along the cycle. The converse
implications are obviously false, e.g., Example~\ref{ex:base1} is an
ACS, which is not irreducible. Regarding (L2), if $\cG$ has no cycle,
then $C$ is nilpotent, so $\lambda(C)=0$. If $\cG$ has a cycle, then
$C^n\not\equiv 0$ for any $n\in \N$. Since the eigenvalues of $C^n$
are $n$-th powers of those of $C$ and $c_{ij}\in\{0,1\}$, it follows that
$\lambda(C)\geq 1$. The first part in (L3) is easy since no cycles means
there is a vertex without an edge pointing to it. If $\lambda(C)\geq 1$, then
we can re-order the vertices such that $\{1,2,\ldots,j_*\}=\{j:v_j>0\}$ and
compute for $i\leq j_*$ that
\benn
0<\lambda v_i = (Cv)_i=\sum_{j=1}^{j_*}c_{ij} v_j.
\eenn   
Hence, $c_{ij}=1$ for at least one index $j$, which means that the graph
induced by $\{c_{ij}\}_{j=1}^{j_*}$ is an ACS.
\end{proof}

The easiest parts of $\cG$ to deal with are graphs with no cycles. Let
\benn
\cT=\{j:\exists i\text{ s.t. }c_{ji}=1, c_{ij}=0~\forall i \}
\eenn
be the set of terminal vertex indices. Re-order vertex labels such that
$\cG$ has as first indices those in $\cT$ so that 
$\cT=\{1,2,\ldots,j^*\}$. Let $e_j$ denote the $j$-th
standard basis vector.

\begin{prop}
\label{prop:steadynocycles}
Suppose $\cG$ has no cycles. Then 
\benn
X^*:=\left\{\sum_{j=1}^{j^*} b_j e_j:\sum_{j=1}^{j^*} b_j=1\right\}
\eenn
are equilibrium points.
\end{prop}

\begin{proof}
Let $x^*\in X^*$. We find $(Cx^*)=0$ since $Cx^*$ only contains entries
for the non-terminal vertices. Therefore, we have $f(x^*)=0$.
\end{proof}

Although $X^*$ contains many equilibrium points, we shall see below 
that not all of them are stable. There is a well-defined
subset of $X^*$, which is going to be stable up to measure-zero initial
data.

\begin{defn}
Suppose $\cG$ has no cycles and terminal vertex set $\cT$. Suppose $j\in \cT$ 
and Let $p(j):=|\{i:\exists\text{ a path from $i$ to $j$}\}|$. Then 
define the maximal input equilibrium points
\benn
X_*:=\left\{\sum_{j=1}^{j^*} b_j e_j:\sum_{j=1}^{j^*} b_j=1, b_j\neq 0 
\text{ iff }p(j)=\max_{k\in\cT} p(k)\right\}.
\eenn
\end{defn}

Clearly, $X_*$ contains only equilibrium points by 
Proposition~\ref{prop:steadynocycles}. The next step is to consider
graphs with cycles. Any graph with at least one cycle has an ACS. The 
previous ideas can be generalized.

\begin{defn}
Suppose $\cG$ has at least one cycle. Let $v_1,\ldots,v_{j^*}$ be the 
set of Perron-Frobenous eigenvectors with eigenvalue $\lambda=\rho(C)$ 
associated to $\cG$. Then define
\benn
X_*:=\left\{\sum_{j=1}^{j_*} b_j v_j:\sum_{j=1}^{j^*} b_j=1,b_j\neq 0
\text{ $j$ is part of an ACS }\right\}.
\eenn
\end{defn}

One checks that $X_*$ contains only equilibrium points. Note that the last 
two definitions imply that $X_*$ is a well-defined for every graph 
$\cG$. $X_*$ is non-empty and we shall show below that it is
an attracting set of equilibrium points up to measure zero of initial 
conditions. 

\subsection{Stability}

In this section, we discuss one possible proof for stability. Some parts
are based upon formal calculations, mentioned in the work of Jain and 
Krishna~\cite{JainKrishna,JainKrishna1}, which we make rigorous here
adding a new geometric viewpoint. Consider the system of ODEs~\eqref{eq:JK}. 
Then consider the mapping
\be
\label{eq:JKtac}
G^{-1}:\R^d\ra \cX,\qquad y\stackrel{G^{-1}}{\mapsto}\frac{y}{|y|_1},\qquad y\in\R^d.
\ee
The inverse mapping $G:\cX\ra \R^d$ is defined via the pre-image so that 
$G$ is a multi-valued mapping, i.e., one should 
view the transformation~\eqref{eq:JKtac} \emph{not} as a coordinate 
change but as an unfolding of the phase space $\cX$ onto a larger phase 
space consisting of the non-negative quadrant/cone
\benn
\R_+^d:=\{x\in\R^d:x_j\geq 0 \text{ for all $j$}\}.
\eenn 
The next proposition shows that we just have a differently scaled 
and restricted version of real projective space.

\begin{prop}
The following results hold:
\begin{itemize}
 \item[(P1)] Fix any vector $v\in\R^d_+$, then all lines with direction $v$ through the origin in $\R_+^d$ 
map to the same point under $G^{-1}$; 
 \item[(P2)] On $\R^d_+$, the ODEs are given by
\be
\label{eq:yJK}
y'=Cy-\phi y,
\ee
valid for any choice of $\phi\in\R$.
\end{itemize}
\end{prop}

\begin{proof}
For (P1), consider the line $y=l_v(a)=a v$ for $v \in[0,\I)^d$, $v\not\equiv 0$, $v$ 
fixed and $a\in (0,\I)$, then we calculate
\benn
G^{-1}(av)=\frac{a v}{|a v|_1}=\frac{v}{|v|_1}=x_v,
\eenn
where $x_v$ is fixed since $v$ is fixed for the line so all points on the line
map to the same point under $G$.
For (P2), note that we have on the cone $\R_+^d$ the simpler version 
of the $1$-norm $|y|_1=\sum_{j=1}^d y_j$ so that
\benn
\frac{\txtd x}{\txtd t}=\frac{y'|y|_1-y|y'|_1}{|y|_1^2}=\frac{Cy-\phi y}{|y|}
-\frac{y \sum_{j=1}^d (Cy)_j-\phi y_j }{|y|_1}=Cx-|Cx|x.
\eenn
The free parameter $\phi$ essentially gives the additional degree-of-freedom
due to the lack of mass conservation in~\eqref{eq:yJK}. 
\end{proof}

Due to (P1), we can view the model~\eqref{eq:JK}-\eqref{eq:JKc} as posed on part
of projective space $\R P^{d-1}=\R^d/(x\sim b x)$ for $b\neq 0$. Indeed, in our
context we just
start with the non-negative quadrant $\R^d$ and then apply the equivalence relation
$x\sim bx$ for $b\neq 0$; note that we use the $1$-norm instead of the usual 
$2$-norm to identify a unique point in $\cX$ but this makes no difference from a 
topological point of view. For stability, the ODEs~\eqref{eq:yJK} can be used.
Let $\mu_\cX$ denote the $(d-1)$-dimensional Lebesgue measure induced on $\cX$.

\begin{thm}
\label{thm:convtosteady1}
Suppose $\cG$ has at least one cycle. For almost every $x_0\in\cX$ (wrt $\mu_\cX$), 
there exists $x_*\in X_*$ such that
\benn
\lim_{t\ra +\I}x(t)=x_*\qquad \text{$x(t)$ solves~\eqref{eq:yJK} with $x(0)=x_0$.}
\eenn 
\end{thm}

\begin{proof}
As before, fix $v\not\equiv 0$, $v\in\R^d_+$, and set $l_v:=\{y\in\R^{d}_+:\exists 
a\geq 0\text{ s.t. }av=y\}$. Let $\mu_Y$ denote the Lebesgue measure induced via 
$G$ on $\R^{d}_+$, i.e.,
\benn
\mu_Y(\cB):=\mu_\cX(\{v\in\cX:l_v\cap \cB\neq \emptyset\}).
\eenn  
Since $\cG$ has at least one cycle, any $x_*$ is a weighted convex combination of 
vectors $\{v_j\}_{j=1}^{j_*}$ in the eigenspace 
associated to the leading Perron-Frobenius eigenvalues. Consider the system~\eqref{eq:yJK} 
and fix any $y_0=y(0)$. Select $\phi<0$ such that all solutions $y(t)$ diverge sufficiently
fast, i.e., $|y|_1\geq p(t)\txte^{kt}$ for all $t\geq t_0>0$ and some polynomial $p(t)$. 
The solution to~\eqref{eq:yJK} can be written as 
\be
\label{eq:Clinsolve}
y(t)=\sum_{j=1}^d p_j(t)\txte^{\tilde{\lambda}_j t}\tilde{v}_j, 
\ee 
where $(\tilde{\lambda_j},\tilde{v}_j)$ are (generalized) eigenpairs for $C-\phi\Id$, 
$\textnormal{Re}(\tilde{\lambda}_j)>0$, and $p_j(t)$ are polynomials in $t$. Note that
up to a re-ordering of variables, we can assume that $\tilde{\lambda}_j$ for 
$j=1,2,\ldots,j_*$ are the leading eigenvalues associated to
the Perron-Frobenius eigenvalues of $C$ as defined in Section~\ref{ssec:steadystates}
by a shift of $-\phi$. In particular, we have 
\be
\label{eq:eigcondcool}
\textnormal{Re}(\tilde{\lambda}_{j_*})>\textnormal{Re}(\tilde{\lambda}_j)\quad 
\forall j\geq j_*.
\ee
Given any $\varepsilon>0$ and $y_0\in \cH:=\{y:y\notin\textnormal{span}(v_{j}:j>j_*)\}$, 
there exists $T$ such that 
\be
\label{eq:Hauscond}
d_H(y(t),\textnormal{span}(v_j:j\leq j_*))<\varepsilon\qquad \forall t> T,
\ee  
where $d_H$ is the usual Hausdorff distance and we used that the strongest expanding 
directions in~\eqref{eq:Clinsolve} eventually dominate any weaker expanding direction, i.e., 
we have
\benn
\lim_{t\ra +\I} \frac{p_{j}(t)\txte^{\tilde{\lambda}_{j}t}}{p_{j_*}(t)\txte^{\tilde{\lambda}_{j_*}t}}=0
\eenn
for any non-trivial polynomials $p_{j}(t)$ and $p_{j_*}(t)$ by~\eqref{eq:eigcondcool}
from which we can conclude~\eqref{eq:Hauscond}. 
Note that we just excluded a set of measure $\mu_Y$ in the last argument as initial 
conditions satisfy $\mu_Y(\cH)=0$ since Lebesgue measure vanishes on subspaces 
of dimension strictly less than the space dimension. Next, we observe that any point in
subspaces contained
\benn 
\textnormal{span}(v_j:j\leq j_*)
\eenn
is associated to a line $l_v$ such that $l_v$ maps to a point $x_*\in X_*$ under 
$G$. Since $\varepsilon>0$ was arbitrary, the result follows.
\end{proof}

The intuition is that the PFEs define the subspace, which is fastest growing so 
all initial conditions except those in non-leading eigenspaces are attracted to 
the span of the PFEs as $t\ra +\I$. To cover the case of no cycles, we start
with a definition:

\begin{defn}
The power-weighted and power-weighted average variables $R_n$ 
and $r_n$ are defined as follows
\benn
R_n:=C^nx\qquad \text{and}\qquad r_n:=\sum_{j=1}^d (C^nx)_j.
\eenn
\end{defn}

\begin{lem}(Adaptive Network Dynamics Intertwining)
\label{lem:ANDI}
The variables $r_n=r_n(t)$ satisfy the infinite-dimensional ODE system
\be
\label{eq:moments}
r_n'=r_{n+1}-r_nr_1\qquad n\in \N.
\ee
The variables $R_n=R_n(t)$ satisfy the infinite-dimensional ODE system
\be
\label{eq:moments1}
R_n'=R_{n+1}-R_n |R_1|_1\qquad n\in \N.
\ee
\end{lem}

\begin{proof}
We just compute
\beann
r_n'&=&\sum_{j=1}^d (C^nx')_j=\sum_{j=1}^d\left(C^{n+1}x-C^nx|Cx|_1\right)_j\\
&=&\sum_{j=1}^d\left(C^{n+1}x\right)_j-\sum_{j=1}^d\left(C^nx\right)_j|Cx|_1\\
&=& r_{n+1}-r_nr_1,
\eeann
which shows~\eqref{eq:moments}. The computation 
for $R_n$ is even easier 
\benn
R_n'=C^nx'=C^{n+1}x-|Cx|_1 C^nx=R_{n+1}-|R_1|_1 R_n,
\eenn
which finishes the proof.
\end{proof}

The Adaptive Network Dynamics Intertwining (ANDI) Lemma~\ref{lem:ANDI} 
connects the topology and cycle structure of the graph with the vertex 
dynamics into a \emph{bigger} dynamical system. Although the proof is very 
simple, the insight is still substantial
as we have connected the dynamics \emph{on the network} modelled 
by~\eqref{eq:JK}-\eqref{eq:JKc} with the dynamics \emph{of the network}
encoded in $C$. Since this general approach to find a dynamical system 
encoding \emph{both components} of an adaptive network is not restricted 
to the Jain-Krishna model, we conjecture there are ANDI Lemma results
for many other types of adaptive networks. Using Lemma~\ref{lem:ANDI}
it is easy to check that certain components must decay to zero.

\begin{prop}
\label{prop:terminalelim}
Suppose $\cG$ has no cycles and $j\not\in \cT$, then $x_j(t)\ra 0$
as $t\ra +\I$. 
\end{prop}

\begin{proof}
First, we can eliminate all vertices $j$, which are not connected to any
other vertices since they satisfy the equation
\benn
x_j'=-x_j|Cx|_1.
\eenn
Similarly, if $\cG$ has multiple connected components, we can 
analyze each component separately so we restrict attention to a
single connected (sub-)graph. Order the indices so that 
$\{1,2,\ldots,j^*\}=\cT$. Since there are no 
cycles, $C$ is nilpotent. Therefore, $R_n=0$ for all 
$n> n_*$ for some $n_*\geq 2$ and so
\beann
R_1'&=& R_2-R_1|R_1|,\\
\vdots &=& \vdots,\\
R_{n_*}'&=& -R_{n_*}|R_1|.\\
\eeann
Therefore, $R_{n_*}\ra 0$ as $t\ra +\I$ and we eventually get by
eliminating everything up to the first equation that $R_1\ra 0$ as
well. However, $R_1$ only contains linear combinations of non-terminal
vertex variables, which proves the result. 
\end{proof}

\begin{thm}
\label{thm:convpart2}
Suppose $\cG$ has no cycles and $\cG$ has at least one edge. For 
almost every $x_0\in\cX$ (wrt $\mu_\cX$), there exists $x_*\in X_*$ such that
\benn
\lim_{t\ra +\I}x(t)=x_*\qquad \text{$x(t)$ solves~\eqref{eq:yJK} with $x(0)=x_0$.}
\eenn 
\end{thm}

\begin{proof}
By Proposition~\ref{prop:terminalelim}, we already know that all components 
associated to non-terminal vertices vanish as $t\ra +\I$, so any accumulation 
point $x_*$ of $\{x(t):t\in[0,\I)\}$ must lie in $X^*$. We still have to prove 
that $x_*\in X_*$, i.e., only those terminal components with the maximum number 
of input paths can remain. Without loss of generality, suppose we just have a 
graph $\cG$ with a single connected component. Since $C$ is the transpose
of the classical adjacency matrix, it follows that 
\benn
(C^n)_{ij} = |\{\textnormal{paths from $j$ to $i$ of length $n$}\}|.
\eenn
Hence, if $\{1,2,\ldots,j^c\}$ with $j^c\leq j^*$ is the set of indices of
terminal vertices with the maximum number of paths pointing into them, we have
\be
\label{eq:condpaths}
\sum_{j=1}^d\sum_{n=0}^\I (C^n)_{i_1j} > \sum_{j=1}^d \sum_{n=0}^\I (C^n)_{i_2j}\quad 
\text{if $i_1\leq j^c$ and $i_2>j^c$.} 
\ee
Of course, the summation over $n$ is actually finite since $C$ is nilpotent as
there are no cycles. Switching again to the formulation of the ODEs~\eqref{eq:yJK}
on projective space we have
\benn
(y(t))_i=\sum_{j=1}^d\sum_{n=0}^\I \frac{t^n}{n!}((C-\phi\Id)^n)_{ij} (y(0))_j.
\eenn
Hence, using that the identity matrix commutes with every matrix, 
applying~\eqref{eq:condpaths}, and picking $\phi$ so that all components grow,
similar to the idea in the proof of Theorem~\ref{thm:convtosteady1}, we see that
those components with $i\leq j^c$ grow fastest so after projection into 
$x$-variables, only those terminal vertices with maximal number of input paths
will remain in the equilibrium point.
\end{proof}

In summary, the result for the ODE dynamics of the Jain-Krishna model is 
relatively simple: generically, we obtain an equilibrium point, which either
is non-trivial and governed by one (or more) ACS, or the graph has no cycles and 
we concentrate on the terminal vertices with maximal input. In the infinite time-scale
separation limit, this essentially means that we can regard the ODE dynamics 
as a discrete map, which produces upon a given input $C$ generically a vector 
$x_*$.

\section{Observations for the Discrete-Time Edge Dynamics}
\label{sec:dis}

In this section, we are going to state several observations and conjectures 
regarding the discrete-time Jain-Krishna graph update rule (or ``JK update'', for short) 
stated in Section~\ref{sec:JKmodel}. In contrast to Section~\ref{sec:cts}, we do not 
attempt to provide full proofs 
of our observations in this section, which is a task left for future work.\medskip

We assume that the time-scale separation is infinite in the
sense that the ODEs in Section~\ref{sec:cts} converge instantaneously to an
equilibrium point $x_*$. Consider the index sets
\benn
\cI_*:=\{i\in\{1,2,\ldots,d\}:(x_*)_i>0\}\quad \text{and}\quad 
\cK_*:=\{i\in\{1,2,\ldots,d\}:(x_*)_i=0\}.
\eenn
We may always assume that $\cI_*$ is non-empty as equilibrium points for the ODEs
treated in Section~\ref{sec:cts} satisfy $|x_*|_1=1$. Furthermore, observe that if $\cK_*$ 
is non-empty then $\cK_*=\cJ_*$, which is a natural starting point. Recall that in this 
case, the JK update eliminates all edges to and from a randomly chosen vertex in $\cK_*$, 
and then inserts new edges to and from all other vertices with a fixed probability $p$.
We are interested in characterizing cycles, respectively also ACS, as they
determine, how many vertices are active, i.e., the number of elements $|\cI_*|$
in the set $\cI_*$. 
 
\subsection{Cycle Distribution for Random Graphs}

Since cycles are a key component of an ACS, let us start with looking at the cycles
for the first graph $\cG[0]=\cG$, which is of Erd{\H{o}}s-R{\'e}nyi type.
We start with the undirected case and denote the ensemble of Erd{\H{o}}s-R{\'e}nyi 
random graphs by $\textnormal{ER}_d(p)$, where $p$ is the probability for each 
edge being present, or by $\textnormal{ER}_d(\theta/d)$ where $\theta:=pd$. 
The degree distribution is well-known~\cite{Bollobas,vanderHofstad} 
\beann
p_k:=\P(\textnormal{degree}(v)=k)&=&\left(\begin{array}{c}d-1 \\ k \end{array}\right)
p^k(1-p)^{d-1-k}\\
&=&\frac{(d-1)!}{k!(d-1-k)!}p^k(1-p)^{d-1-k}.
\eeann
The binomial distribution $\textnormal{Bin}(d,\theta/d)$ converges to the 
Poisson distribution $\textnormal{Poi}(\theta)$
\benn
\lim_{d\ra \I} \P\left(\textnormal{Bin}(d,\theta/d)=k\right)
=\txte^{-\theta}\frac{\theta^k}{k!},
\eenn
for each $k\in\N_0$. So $p_k=\txte^{-\theta}\theta^k/(k!)$ is an approximation
for the degree distribution of very large Erd{\H{o}}s-R{\'e}nyi graphs. Hence,
one expects that the Poisson distribution should also appear
in the context of the random variable $\cC_k$ counting the number of cycles 
of length $k$ in $\textnormal{ER}_d(\theta/d)$. Let $k\geq 3$, consider $\theta>0$ 
and define $\mu:=\theta^k/(2k)$. Then one can obtain~\cite{Bordenave} convergence 
in distribution
\benn
\cC_k\stackrel{\txtd}{\ra} \textnormal{Poi}(\mu)=:c_k.
\eenn  
One nice proof uses Stein's method as discussed in~\cite{Bordenave}. For directed
graphs, we note that for a given undirected cycle of length $k$, there are $2^k$
possible edge orientations possible and only two form a directed cycle, so the
distribution is just given by
\benn
\tilde{c}_k=\frac{c_k 2^{k-1}}{\sum_{k=3}^\I c_k 2^{k-1}},\qquad \text{for $k\geq 3$.}
\eenn   
We can also include directed two-cycles by noticing that we need two consecutive
successful binomial trials to generate a two-cycle. Although it is nice to know,
how likely cycles are for the first graph, and although it is evident that
\benn
\P(\textnormal{no cycles})\ra 0\qquad \text{as $d\ra \I$,}
\eenn
there is clearly often the situation that either no cycles, or very few cycles, or just
a very small ACS, appear in a graph $\cG[s]$
in the Jain-Krishna model. Indeed, the Poisson distribution decays very rapidly and 
because even very large ACS/cycle structures can be destroyed if all edges are part
of an ACS/cycle, we have to consider the time scale it takes to generate a cycle by
the Erd{\H{o}}s-R{\'e}nyi graph sampling construction viewed as a dynamical process
in discrete time $s\in\N_0$ (recall we re-sample the transpose of the adjacency matrix
at each time step $s$ for a single vertex). 

\subsection{The First Cycle}

Fix a starting time, wlog $s=0$ and recall that $\cG=\cG[0]$ is fully characterized by 
the transpose of the adjacency matrix $C[0]$. Suppose there is no undirected 
cycle at $s=0$. 
In this section, we just consider the formation of a cycle in the undirected case
but we expect that similar results/conjectures hold upon suitable modification also
for the directed case, i.e., although the graph is directed, we are only interested in
the formation time of the first undirected cycle for simplicity. Since the 
re-sampling at each discrete time step is according to an Erd{\H{o}}s-R{\'e}nyi-type 
rule, we have to examine cycle formation for this rule in further detail.

Recall that the mean degree of $\textnormal{ER}_d(p)$ is $dp$. There is a phase 
transition~\cite{ErdosRenyi} at $dp=1$ leading to a giant component for $dp>1$.
Assume that the undirected version $\cH$ of $\cG$ has an Erd{\H{o}}s-R{\'e}nyi graph 
structure $\textnormal{ER}_d(p)$ and that $\cH$ has no cycle.\medskip 

\textbf{Conjecture 1: If $dp>1$, then the formation of a cycle according to the 
Jain-Krishna update requires $\cO(1)$ (as $d\ra \I$ with $dp$ kept constant) 
time steps independent of any other parameters.}\medskip

Roughly speaking, the conjecture states that cycle formation for $dp>1$ is
extremely \emph{fast}. Let us provide a heuristic to motivate the conjecture. 
Consider deleting and then adding a single vertex to $\cH$ with edge probability to 
other vertices given by $p$. To form a cycle it is most likely that the cycle 
has at least one vertex part of the giant component. Then we have
\benn
\P(\text{``formation of at least one new cycle''})\approx 1-\sum_{j=0}^1 
\left(\begin{array}{c}d-1 \\ j\end{array}\right) p^j(1-p)^{d-1-j},
\eenn
where the second term arises as it is the probability that we only 
connected at most one edge to the giant component. This means in the usual
Poisson approximation we get
\benn
\P(\text{``formation of a at least one new cycle''})\approx 1-\txte^{-dp}(1+dp).
\eenn
Hence, we get the expected result that for a giant component regime, we just
need very few trials to generate a cycle. Since in the Jain-Krishna model, only
directed cycles are considered, a few additional re-samplings might be required
but we shall not discuss this issue here. It is more interesting to look at
the case $dp<1$, since in this case, the approximation of all components by 
the single giant component is clearly wrong.\medskip

\textbf{Conjecture 2: If $dp<1$, then the number of steps $s_*$ to generate
the first cycle by the Jain-Krishna rule is $\cO(d/p)$.}\medskip

The last conjecture says that it may take a long time until a first cycle is 
formed. Again we provide a heuristic argument. Suppose $\cH$ 
has $d$ vertices. Since there is no giant component, forming small loops with 
existing edges is very unlikely. Therefore, the cycle formation problem
can be approximated by the problem to form the \emph{first cycle} in an empty graph $\cH$,
which is filled in each time step with edges according to the JK update.

Instead of the JK update, we break the problem into smaller steps just looking at
adding single edges in each time step. 
Enumerate all possible pairs of edges $e_{ij}$ between vertices $i$ and $j$ with 
$1\leq i\leq j \leq d$. Now
we go through these pairs iteratively and decide at each step with probability $p$,
to add the edge to the graph. This is the classical \emph{permutation model}
of the Erd{\H{o}}s-R{\'e}nyi graph $\textnormal{ER}_d(p)$. Another possible and
related construction is the \emph{uniform model}, where $i$ and $j$ are drawn 
uniformly from $\{1,2,\ldots,d\}$ and the edge $e_{ij}$ is added with probability $p$.
Of course, this means self-loops and double edges can appear generating a random
multigraph, which has actually very similar properties to the random graph
generated via the permutation model. 

It is now a \emph{classical question}, how the first cycle appears in these constructions and 
what its expected length is. In fact, this line of research goes back directly to the 
original work by Erd{\H{o}}s and R{\'e}nyi \cite{ErdosRenyi} and has been studied by a
variety of combinatorial~\cite{FlajoletKnuthPittel} and probabilistic 
techniques~\cite{BollobasRasmussen,Janson}.\footnote{Interestingly, even the same idea of
using catalytic feedback loops to describe pre-biotic evolution appears in the context 
of these classical works on first cycle generation~\cite{BollobasRasmussen} so this 
can be viewed as an early variant the Jain-Krishna model already.} 

It is known that for the uniform and the permutation model, the expected length
of this first cycle is $\cO(d^{1/6})$ as $d\ra \I$ with standard deviation
$\cO(d^{1/4})$. For us, it is more interesting to estimate the time it takes to
generate the first cycle. The expected number of edges when the first cycles 
appears is~\cite{FlajoletKnuthPittel} 
\be
\label{eq:scale1G}
\frac13 d + \cO(d^{5/6})\qquad \text{as $d\ra \I$}
\ee  
for the uniform model, while it is 
\be
\label{eq:scale2G}
\frac12 (1-p_*)d + \cO(d^{5/6})\qquad \text{as $d\ra \I$}
\ee  
for a constant $p_*\approx 0.12$ in the permutation model. Since we add one edge
per step at probability $p$, there are an average of $ps_*$ edges after $s_*$ steps
for each model. Discarding the higher-order terms we get that the condition
\benn
\frac13 d \stackrel{!}{=} ps_* \quad \Rightarrow \quad s_*=\frac13 \frac{d}{p}
=\frac13 d^2/\theta
\eenn
is conjectured to give a good approximation to obtain a first cycle in $s_*$ steps
of the uniform model. A similar condition appears for the permutation model so
we always have $s_*=\cO(d^2/\theta)$, where obviously $\theta=\theta(d)$ may
depend upon the size of the graph in most applications. 

Applying these results to the Jain-Krishna model is not immediately possible 
since (a) we have a directed graph in this context, (b) we do not allow for loops, 
(c) we add not one but potentially many edges at each time step, and (d) we do not 
generate the graph completely but only re-sample one vertex at each time step. 

Scalings with leading-order term $\cO(d)$ as 
in~\eqref{eq:scale1G}-\eqref{eq:scale2G} also hold for directed 
graphs~\cite{BollobasRasmussen}, which means we can conjecture that (a) does not
play a crucial role. The case of loops is also excluded in the directed context
in~\cite{BollobasRasmussen} not affecting the scaling so (b) can also be disregarded.
The problem (c) is essentially just a time re-scaling for large graphs as instead 
of counting one edge at a time, we add a certain average number of edges $2p(d-1)$
at once. However, the problem (d) could be substantial as we have a lot 
more edges available already as we do not start from an empty graph but this 
should be covered by having a very sparse graph upon using $dp<1$. 

\subsection{Generating Large ACS}
\label{ssec:largerACS}

Now suppose $\cG[0]=\cG$ has at least one cycle $\cY\subseteq \cG$ and suppose 
$\cK_*$ is non-empty. Obviously, the cycle only consists of edges in $\cI_*$ 
so it will not be destroyed by the JK update as long as $\cK_*$ 
is non-empty. We are now interested in, how long it is going to take to form
one large ACS $\cZ$. As before, we expect that the size of $dp$ is crucial to 
distinguish several cases. Again, we restrict to the undirected case and 
ask the simpler question, how long it is going to take to form one large 
cycle.\medskip

\textbf{Conjecture 3: If $dp>1$ and we start from a typical cycle, then the 
formation of a single ACS according to the JK update requires at most $\cO(1)$ 
time steps independent of any other parameters.}\medskip

The reasoning for this conjecture is that there will be relatively large cycles
in the case of a giant component and then just the same reasoning as for 
Conjecture 1 applies, so we shall not discuss the reasoning in additional detail. 
As before, the interesting case seems to be $dp<1$.\medskip 

\textbf{Conjecture 4: If $dp<1$ is fixed and $d\ra +\I$, the mean waiting time 
until an ACS has formed containing all vertices is $\cO(d)$.}\medskip

Again we give some heuristics to motivate the conjecture. Suppose our current
ACS $\cZ[0]=\cZ$ in $\cG[0]$ has $k$ vertices. To attach a vertex $v$ via the 
JK update requires the generation of at least one edge from 
$\cZ$ to $v$. The probability of this event is
\benn
\sum_{i=1}^k\left(\begin{array}{c}k \\i\end{array}\right) p^i(1-p)^{k-i}=
1-(1-p)^k=:r(k,p)=r.
\eenn 
Therefore, we have for the mean waiting time $\tau_k$ until the update 
rule has attached a vertex $v$ that is repeatedly selected
\benn
\E[\tau_k]=\sum_{j=1}^\I jr(1-r)^{j-1}=r\sum_{j=1}^\I j (1-r)^{j-1}
=\frac{r}{(1-(1-r))^2}=\frac{1}{r},
\eenn
where the usual differentiation of the geometric series has been used. Now
summing over all the waiting times gives the total waiting time 
\benn
\sum_{k=1}^d \E[\tau_k]=\sum_{k=1}^d \frac{1}{1-(1-p)^k}
\eenn
The largest term in this sum is actually the first one as one 
would intuitively expect as it is easier to attach to an ACS if the ACS is
already large. However, there is also the effect of $d$, particularly for
large graphs. We just calculate
\beann
\sum_{k=1}^d \frac{1}{1-(1-p)^k} &\approx& \int_1^d \frac{1}{1-(1-p)^x}~\txtd x\\
&=& \int_1^d \frac{1}{1-\txte^{x\ln (1-p)}}~\txtd x = \left.
x-\frac{1-\txte^{x\ln(1-p)}}{\ln(1-p)}\right|_{x=1}^d\\
&=& d-1 - \frac{1-(1-p)^d}{\ln(1-p)}  + \frac{p}{\ln(1-p)} 
\eeann
So if $p$ is small and $d$ is large while $dp<1$ is fixed in the 
asymptotic limit $p\ra 0$, we easily check using L'H{\^o}pital's rule that the last 
expression diverges like $1/p$ as $p\ra 0$. Since $\theta=pd<1$ is fixed, we 
have now motivated our Conjecture 4 as the divergence is $\cO(d)$ as $d\ra +\I$. 
Therefore, it takes very long to really form a full ACS activating every vertex, even 
if we start with a cycle already in the first graph. Yet, it does not take as long 
as forming cycles in the first place since we essentially form an ACS in a more 
deterministic way by judiciously eliminating vertices not part of an ACS at 
each step. 

\section{Summary and Outlook}

In this work, we have analyzed the two singular limits of the Jain-Krishna model
for adaptive catalytic networks. One limit consists in freezing the dynamics
of the network working on a fixed graph. In this case, we have proven the existence 
of stationary solutions for the ODE vertex dynamics. We have characterized the equilibria 
via Perron-Frobenius vectors, including a distinction between different cycle structures 
and autocatalytic sets of the underlying graph. Then we 
have rigorously proven that the dynamics of the catalytic ODEs converges (up to 
initial conditions of measure zero) to a stationary solution. The proof uses the insight
that working in a projective space is the correct mathematical setting and it uses
an intertwining lemma, which yields an infinite-dimensional system of ODEs including 
the graph structure so that the dynamics can be treated hierarchically. Then we studied
the second singular limit, where the vertex dynamics are assumed to be infinitely fast,
yet the dynamics of the network yields a changing graph. In this context, we formulated
four conjectures, which we justified with formal asymptotic calculations for the two
parameters $d$ (number of vertices) and $p$ (connection probability between vertices). 
More precisely there seems to be a substantial distinction for the Jain-Krishna model 
between the case, when the graph has a giant component ($dp>1$) and when it does not 
($dp<1$). More precisely, in the giant component case, we conjectured that the formation 
of an initial cycle, as well as the formation of an autocatalytic set, are of orders
$\cO(d^2)$ and $\cO(d)$ respectively as $d\ra +\I$. In the case without a giant 
component, these events are conjectured to occur typically at order $\cO(1)$. 

In summary, we can conclude that for the non-giant component case $dp<1$, the dynamics
is a true multiple time scale dynamical system: the ODE convergence of the vertex dynamics 
to equilibrium is generically exponentially fast, and the network is slowly driven by the 
adaptation of the graph at a very slow time scale of at least $\cO(d)$. Yet, if we are 
very close to, or even have $dp>1$, we expect a mixing of the time scales.\medskip 

The critical regime $dp\approx 1$ seems to
be particularly interesting as there could be links to self-organized criticality (SOC) in
networks~\cite{BornholdtRohlf,ChristensenDonangeloKoillerSneppen,PaczuskiBasslerCorral},
where time-scale separation again plays a crucial role~\cite{KuehnNetworks}. SOC exploits
that the network is near a topologically critical point, e.g., being marginally connected,
to enhance information processing. From the viewpoint of applications, one may hence
infer that it could be beneficial to study the Jain-Krishna model also very close to the
critical regime $dp\approx 1$, where it can change from a multiple time scale network
to a more single-scale structure between ODE vertex dynamics and network adaption.\medskip

For the regime $dp<1$, it seems to be a natural next step to exploit the scale
separation to make predictions about the dynamics. For example, we know that an
autocatalytic set (ACS) can form only very slowly. The ACS can only be destroyed 
once it encompasses the entire graph, which means that we may be able to 
derive early-warning signs for drastic sudden transitions in such networks. Indeed,
this has been shown to be possible in several other adaptive network models already~\cite{KuehnZschalerGross,HorstmeyerKuehnThurner,Cavaliereetal,JaegerHoferKapellerFuellsack}.\medskip

From the standpoint of rigorous mathematical analysis, it seems plausible that it
might be possible to give a fully rigorous analysis of certain patterns observed
in time series of the Jain-Krishna model. At least for the case $dp<1$ one may
use the equilibrium point results we proved in Section~\ref{sec:cts} and try to 
establish Conjectures 2 \& 4. Then another step employing (singular) perturbation 
theory should yield the existence of long phases of slow network adaptation towards 
an ACS followed by a sudden collapses, which is the main phenomenon exhibited by 
the Jain-Krishna model. The work presented here provides at least a first step in
this, potentially very involved, path towards a complete mathematical analysis.

\appendix 

\section{An Inconsistent Model Definition}
\label{ap:bug}

In this appendix, we prove that the definition of the model in the 
paper~\cite{JainKrishna3} is incorrect. We briefly repeat the definition
from~\cite{JainKrishna3}. Given $C\in\R^{d\times d}$, we start by sampling 
an interaction matrix $C=(c_{ij})$ so that 
\benn
\P(c_{ij}\neq 0)=p\qquad\text{and} \qquad \P(c_{ij}= 0)=1-p 
\eenn
for some fixed $p\in(0,1)$. Then the entry $c_{ij}$ is determined by 
\emph{sampling uniformly from $[-1,1]$ if $i\neq j$ and from $[-1,0]$ 
if $i=j$}, which is the first main difference to 
Section~\ref{sec:JKmodel}. Then we set as above
\be
\label{eq:f}
f_i:=(Cx)_i-x_i\sum_{k=1}^d (Cx)_k
\ee
but define the ODEs for $x=x(t)\in\R^d$ as
\be
\label{eq:ODE}
x_i'=\left\{
\begin{array}{ll}
f_i & \text{if $x_i>0$ or $f_i\geq 0$,}\\
0 & \text{if $x_i=0$ and $f_i< 0$.} 
\end{array}
\right.
\ee
Furthermore, we impose the constraints
\be
\label{eq:constraints}
0\leq x_i\leq 1,\qquad \sum_{i=1}^d x_i=1.
\ee
defining the phase space $\cX$.

\begin{thm}
The model definition~\eqref{eq:f}-\eqref{eq:constraints} is consistent
in the interior $\textnormal{int}(\cX)$, yet it is inconsistent, i.e., 
not solvable, once points on $\partial \cX$ are considered.
\end{thm}

\begin{proof}
Let us start with the interior, there the proof is easy since we only have
to check whether the conservation of mass
\be
\label{eq:mass}
\sum_{j=1}^d x_j=1
\ee
is consistent with~\eqref{eq:ODE}. Indeed, this is easy since differentiating
\eqref{eq:mass} we get
\benn
0=\frac{\txtd}{\txtd t}\left(\sum_{j=1}^d x_j\right)=\sum_{j=1}^d f_j
\eenn
as we are in the interior. So we find
\benn
0=\sum_{j=1}^d f_j = \sum_{j=1}^d(Cx)_j-\underbrace{\sum_{j=1}^d x_j}_{=1}
\sum_{k=1}^d (Cx)_k= \sum_{j=1}^d(Cx)_j-\sum_{k=1}^d (Cx)_k=0.
\eenn
Hence, the model is consistent in the sense that the \emph{physical 
conservation law is also enforced by the ODEs themselves}. This means 
we actually do not really need the conservation law in 
$\textnormal{int}{\cX}$. However, on $\partial\cX$
the same calculation fails. Consider a point with $x_r=0$ and $x_j>0$ for 
$j\neq r$ with $f_r<0$; it is easy to see that we can reach such a point in 
finite time for an open set of matrices $C$, i.e., this is not 
a special case. Now we compute 
\beann
0&=&\sum_{j=1}^d x_j'= \sum_{j=1,j\neq r}^d f_j\\
&=& \sum_{j=1,j\neq r}^d(Cx)_j-\underbrace{\sum_{j=1,j\neq r}^d x_j}_{=1}
\sum_{k=1}^d (Cx)_k = 
\sum_{j=1,j\neq r}^d(Cx)_j-\sum_{k=1}^d (Cx)_k
\eeann
and the last expression on the right-hand side is in general not equal to zero
for some open set of possible matrices $C$. 
\end{proof}

In fact, one can easily see that there are cases, where trajectories
leave $\cX$, i.e., the ODEs are inconsistent with the mass constraint. 
Hence, one should follow the model definitions 
in \cite{JainKrishna,JainKrishna1} instead of~\cite{JainKrishna3}.
 


\begin{thebibliography}{10}

\bibitem{ArnoldSDE}
L.~Arnold.
\newblock {\em Random Dynamical Systems}.
\newblock Springer, Berlin Heidelberg, Germany, 2003.

\bibitem{Bollobas}
B.~Bollob{\'a}s.
\newblock {\em Random Graphs}.
\newblock CUP, 2001.

\bibitem{BollobasRasmussen}
B.~Bollob{\'a}s and S.~Rasmussen.
\newblock First cycles in random directed graph processes.
\newblock {\em Discrete Math.}, 75(1):55--68, 1989.

\bibitem{Bordenave}
C.~Bordenave.
\newblock Random graphs and probabilistic combinatorial optimization.
\newblock {\em Lecture Notes}, pages 1--99, 2016.

\bibitem{BornholdtRohlf}
S.~Bornholdt and T.~Rohlf.
\newblock Topological evolution of dynamical networks: global criticality from
  local dynamics.
\newblock {\em Phys. Rev. Lett.}, 84(26):6114--6117, 2000.

\bibitem{Cavaliereetal}
M.~Cavaliere, G.~Yang, V.~Danos, and V.~Dakos.
\newblock Detecting the collapse of cooperation in evolving networks.
\newblock {\em Sci. Rep.}, 6:30845, 2016.

\bibitem{ChristensenDonangeloKoillerSneppen}
K.~Christensen, R.~Donangelo, B.~Koiller, and K.~Sneppen.
\newblock Evolution of random networks.
\newblock {\em Phys. Rev. Lett.}, 81(11):2380--2383, 1998.

\bibitem{Eigen}
M.~Eigen.
\newblock Molecular self-organization and the early stages of evolution.
\newblock {\em Q. Rev. Biophys.}, 4(2):149--212, 1971.

\bibitem{EigenSchuster}
M.~Eigen and P.~Schuster.
\newblock {\em The Hypercycle: A Principle of Natural Self-organization}.
\newblock Springer, 2012.

\bibitem{ErdosRenyi}
P.~Erd{\"o}s and A.~R{\'e}nyi.
\newblock On the evolution of random graphs.
\newblock {\em Publ. Math. Inst. Hung. Acad. Sci.}, 5(1):17--60, 1960.

\bibitem{Fenichel4}
N.~Fenichel.
\newblock Geometric singular perturbation theory for ordinary differential
  equations.
\newblock {\em J. Differential Equat.}, 31:53--98, 1979.

\bibitem{FlajoletKnuthPittel}
P.~Flajolet, D.E. Knuth, and B.~Pittel.
\newblock The first cycles in an evolving graph.
\newblock {\em Discrete Math.}, 75(1):167--215, 1989.

\bibitem{Gros}
C.~Gros.
\newblock {\em Complex and Adaptive Dynamical Systems}.
\newblock Springer, 2008.

\bibitem{GrossDLimaBlasius}
T.~Gross, C.J.~Dommar D'Lima, and B.~Blasius.
\newblock Epidemic dynamics on an adaptive network.
\newblock {\em Phys. Rev. Lett.}, 96:(208701), 2006.

\bibitem{GrossSayama}
T.~Gross and H.~Sayama, editors.
\newblock {\em Adaptive Networks: Theory, Models and Applications}.
\newblock Springer, 2009.

\bibitem{HordijkSteel}
W.~Hordijk and M.~Steel.
\newblock Detecting autocatalytic, self-sustaining sets in chemical reaction
  systems.
\newblock {\em J. Theor. Biol.}, 227(4):451--461, 2004.

\bibitem{HorstmeyerKuehnThurner}
L.~Horstmeyer, C.~Kuehn, and S.~Thurner.
\newblock Network topology near criticality in adaptive epidemics.
\newblock {\em Phys. Rev. E}, 98:042313, 2018.

\bibitem{JaegerHoferKapellerFuellsack}
G.~J{\"{a}}ger, C.~Hofer, M.~Kapeller, and M.~F{\"u}llsack.
\newblock Hidden early-warning signals in scale-free networks.
\newblock {\em PloS one}, 12(12):e0189853, 2017.

\bibitem{JainKrishna2}
S.~Jain and S.~Krishna.
\newblock Autocatalytic sets and the growth of complexity in an evolutionary
  model.
\newblock {\em Phys. Rev. Lett.}, 81(25):5684--5687, 1998.

\bibitem{JainKrishna1}
S.~Jain and S.~Krishna.
\newblock Emergence and growth of complex networks in adaptive systems.
\newblock {\em Comput. Phys. Comm.}, 121:116--121, 1999.

\bibitem{JainKrishna3}
S.~Jain and S.~Krishna.
\newblock A model for the emergence of cooperation, interdependence, and
  structure in evolving networks.
\newblock {\em Proceed. Natl. Acad. Sci. USA}, 98(2):543--547, 2001.

\bibitem{JainKrishna}
S.~Jain and S.~Krishna.
\newblock Graph theory and the evolution of autocatalytic networks.
\newblock In S.~Bornholdt and H.G. Schuster, editors, {\em Handbook of Graphs
  and Networks: From the Genome to the Internet}, pages 355--395. Wiley, 2006.

\bibitem{Janson}
S.~Janson.
\newblock Poisson convergence and {Poisson} processes with applications to
  random graphs.
\newblock {\em Stoch. Proc. Appl.}, 26:1--30, 1987.

\bibitem{Jones}
C.K.R.T. Jones.
\newblock Geometric singular perturbation theory.
\newblock In {\em Dynamical Systems (Montecatini Terme, 1994)}, volume 1609 of
  {\em Lect. Notes Math.}, pages 44--118. Springer, 1995.

\bibitem{Kauffman1}
S.A. Kauffman.
\newblock Cellular homeostasis, epigenesis and replication in randomly
  aggregated macromolecular systems.
\newblock {\em Cybernetics and Systems}, 1(1):71--96, 1971.

\bibitem{Kauffman}
S.A. Kauffman.
\newblock Autocatalytic sets of proteins.
\newblock {\em J. Theor. Biol.}, 119(1):1--24, 1986.

\bibitem{KuehnNetworks}
C.~Kuehn.
\newblock Time-scale and noise optimality in self-organized critical adaptive
  networks.
\newblock {\em Phys. Rev. E}, 85(2):026103, 2012.

\bibitem{KuehnBook}
C.~Kuehn.
\newblock {\em Multiple Time Scale Dynamics}.
\newblock Springer, 2015.
\newblock 814 pp.

\bibitem{KuehnZschalerGross}
C.~Kuehn, G.~Zschaler, and T.~Gross.
\newblock Early warning signs for saddle-escape transitions in complex
  networks.
\newblock {\em Scientific Reports}, 5:13190, 2015.

\bibitem{LindstroemPettersson}
B.~Lindstr{\"o}m and L.J. Peterrsson.
\newblock A brief history of catalysis.
\newblock {\em Cattech}, 7(4):130--138, 2003.

\bibitem{PachecoTraulsenNowak}
J.M. Pacheco, A.~Traulsen, and M.A. Nowak.
\newblock Coevolution of strategy and structure in complex networks with
  dynamical linking.
\newblock {\em Phys. Rev. Lett.}, 97:(258103), 2006.

\bibitem{PaczuskiBasslerCorral}
M.~Paczuski, K.E. Bassler, and A.~Corral.
\newblock Self-organized networks of competing boolean agents.
\newblock {\em Phys. Rev. Lett.}, 84(14):3185--3188, 2000.

\bibitem{ShawSchwartz}
L.B. Shaw and I.B. Schwartz.
\newblock Fluctuating epidemics on adaptive networks.
\newblock {\em Phys. Rev. E}, 77:(066101), 2008.

\bibitem{Tikhonov}
A.N. Tikhonov.
\newblock Systems of differential equations containing small small parameters
  in the derivatives.
\newblock {\em Mat. Sbornik N. S.}, 31:575--586, 1952.

\bibitem{vanderHofstad}
R.~van~den Hofstad.
\newblock {\em Random Graphs and Complex Networks: Volume 1}.
\newblock CUP, 2016.

\end{thebibliography}
\end{document}